\documentclass[final]{siamltex}

\usepackage{epsfig}
\usepackage{amsmath,amsfonts,amssymb}
\def\setC{\,\mathbb{C}}
\def\setR{\mathbb{R}}

\def\als{{\scriptscriptstyle\#}}
\def\alp#1{{#1_\als}}

\newcommand{\R}{{\mathbb R}}
\newcommand{\C}{{\mathbb C}}

\newcommand{\Cnxn}{\C^{n \times n}}
\newcommand{\Rnxn}{\R^{n \times n}}
\newcommand{\set}[2]{ \left\{ {#1}\,\big\vert\, {#2}\right\}}
\newcommand{\tr}{\operatorname{tr}}
\newcommand{\inprod}[1]{\left\langle #1 \right\rangle}
\newcommand{\norm}[1]{\Vert #1 \Vert}

\newtheorem{rumatemppu}{Example}
\newenvironment{example}{\begin{rumatemppu}\rm}{\end{rumatemppu}}
\newtheorem{rumatemppuu}{Algorithm}

\title{%
ORTHOGONAL POLYNOMIALS OF 
THE $\mathbb{R}$-LINEAR GENERALIZED MINIMAL RESIDUAL METHOD}
\author{
Marko Huhtanen\thanks{
Department of Mathematics and Systems Analysis,
Aalto University, 
P.O.Box 11100
FI-00076 Aalto,
Finland,
({\tt Marko.Huhtanen@tkk.fi}). Supported by the Academy of Finland.}
\and 
Allan Per\"am\"aki \thanks{
Department of Mathematics and Systems Analysis,
Aalto University, 
P.O.Box 11100
FI-00076 Aalto,
Finland,
({\tt Allan.Peramaki@tkk.fi}).
} 
}

\begin{document}
\maketitle
\begin{abstract} 
The speed of convergence of the $\mathbb{R}$-linear GMRES 
method is bounded in terms of a
polynomial approximation problem on a finite subset of the spectrum. 
This result resembles the classical GMRES convergence estimate
except that the matrix involved is assumed to be condiagonalizable.
The bounds obtained are applicable to the CSYM method, in which case they
are sharp. Then
a three term recurrence for generating a family of 
orthogonal polynomials is shown to exist, yielding a natural link with complex
symmetric Jacobi matrices. 
This shows that a mathematical framework analogous to the one
appearing with the Hermitian Lanczos method
exists in the complex symmetric case.
The probability of being  condiagonalizable is 
estimated with random matrices.
\end{abstract}
\begin{keywords} $\mathbb{R}$-linear GMRES, condiagonalizable, 
orthogonal polynomial, Jacobi matrix,
spectrum, polynomial approximation, CSYM, three term recurrence, random matrix
\end{keywords}

\begin{AMS} 
65F15, 42C05
\end{AMS}

\pagestyle{myheadings}
\thispagestyle{plain}

\markboth{M. HUHTANEN AND A. PER\"AM\"AKI   
}{POLYNOMIALS OF THE $\setR$-LINEAR GMRES} 

\section{Introduction}  
Suggested in \cite{EHV},  there exists an
$\setR$-linear 
GMRES (generalized minimal residual) 
 method for solving 
a large real linear
system of equations of the form
\begin{equation}\label{sys}
\kappa z +\alp{M}\overline{z}=b
\end{equation}
for $\kappa \in \setC$, $\alp M\in \setC^{n \times n}$ and $b \in \setC^n$.
Systems of  this type appear regularly in applications.
This is manifested by the complex symmetric case
which corresponds to $\kappa=0$ and $\alp M^T=\alp M$. 
(For its importance in applications, 
such as the numerical solution of the complex Helmholtz equation,
see \cite{FRE}.)
Then the $\setR$-linear 
GMRES method reduces to the CSYM method \cite{BUN}.
We have $\kappa \not=0$, e.g., 
in an approach to solve the electrical conductivity problem \cite{AMPP}
which requires solving an $\setR$-linear Beltrami equation
\cite{HP}. For a wealth of information regarding
real linearity, see \cite{HR,DS}.
Although the $\setR$-linear GMRES method
is a natural scheme, its properties  are  
not well understood.
Assuming $\alp M$ to be condiagonalizable, 
in this paper  a polynomial approximation problem
on the plane is introduced for assessing its speed of convergence.
In the complex symmetric case 
a three term recurrence for generating  
orthogonal polynomials arises, leading to a natural link with complex
symmetric Jacobi matrices. 

The bounds obtained are intriguing by the fact that
they show that the convergence depends on the spectrum 
of the real linear operator involved.
So far it has not been clear what is the significance of the spectrum 
in general and for iterative methods  in particular \cite{HJ1,EHV}. 
Here it is
shown to play a role similar to what the spectrum does in 
the classical GMRES bounds \cite{SASU}. A striking difference is
that the bounds reveal a strong dependence of the  
speed of convergence on the vector. 

Moreover, with any natural Krylov subspace method there exists a connection between 
the iteration and orthogonal functions. As a rule, these are
associated with normality. 
The Hermitian Lanczos method is related with a three term recurrence for
generating orthogonal polynomials; see \cite{GME} and references therein. 
For unitary matrices the corresponding length of recurrence is five \cite{RAG};
see also \cite{SIB}. 
These are special instances of the general framework 
for normal matrices\footnote{The length of recurrence depends
on what is the least possible degree for an algebraic curve 
to contain the eigenvalues.} described in \cite{HUCO,HULA}. 
In this paper an analogous 
connection is established in the complex symmetric case
to orthogonalize monomials 
\begin{equation}\label{mono}
1,\, \lambda,\, |\lambda|^2,\, \lambda|\lambda|^2,\,
|\lambda|^4, \, \lambda|\lambda|^4, \ldots 
\end{equation}
with a three term recurrence. 
This link is not entirely unexpected
by the fact that antilinear operators
involving a complex symmetric matrix $\alp M$ have been regarded as 
yielding  an analogue of normality \cite[p. 250]{HJ1}.
As opposed to the Hermitian Lanczos method, the structure is richer now 
as orthogonality based on a three term 
recurrence and respective rapid least squares approximation 
is possible on very peculiar curves in $\setC$ 
(and not just on subsets of $\setR$ which are also admissible); 
see the assumptions of Theorem \ref{weier} for the admissible curves.
The arising family of functions can be
viewed to extend radial functions in a natural way.

Unlike diagonalizability in the complex linear case,
condiagonalizability is a more intricate structure.
Random matrix theory is invoked to assess how likely it is to have 
a condiagonalizable operator in \eqref{sys}. In this manner we end 
up touching many aspects of the theory that 
has been linked by the classical Hermitian Lanczos method
in recent years \cite{DEIFT}.

The paper is organized as follows. In Section \ref{luku2} 
bounds on the $\setR$-linear GMRES convergence are derived
in the condiagonalizable case. 
The probability 
of a matrix being condiagonalizable is assessed in Section 
\ref{luku3}. Section \ref{luku4} is concerned with the theory of 
orthogonal polynomials related with the 
$\setR$-linear Arnoldi method.
It is shown that complex symmetry is naturally 
treated within antilinear 
structure. Only then its rich properties become visible.
In Section \ref{luku5} some preliminary numerical experiments
are presented.

\section{Condiagonalizability and the convergence of the 
$\setR$-linear GMRES}\label{luku2}  
Condiagonalizability means that the real linear operator
appearing on the left-hand side of \eqref{sys} is diagonalizable. 
Before deriving the bounds,
we first recall how Krylov subspaces are generated
with an $\setR$-linear operator in \eqref{sys}
 by executing the $\setR$-linear
Arnoldi method.

\subsection{Krylov subspaces of the $\setR$-linear GMRES}  
When  $\setC^n$ is regarded as a vector space over $\setC$,
any real linear operator  can be presented as
\begin{equation}\label{remap}
z \longmapsto \mathcal{M}z=(M  +\alp M\tau)z=
Mz +\alp M\overline{z} 
\end{equation}
with  matrices $M, \alp M\in\setC^{n \times n}$. 
Here  %
$\tau$ denotes the conjugation operator 
on $\setC^n$. 
The set of eigenvalues, i.e., the spectrum of a real
linear operator $\mathcal{M}=M  +\alp M\tau$ is defined as
$$\set{\lambda \in \setC}{ \mathcal{M}z=\lambda z \,
\mbox{ for some }\, z\not=0}.$$
The spectrum is an algebraic set
of degree $2n$ at most. 
For more details on the real linear eigenvalue problem, 
see \cite{EHV,HPf}. 

In this paper we are interested in having $M=\kappa I$ 
for a scalar $\kappa\in \setC$. Then the real linear operator
is denoted by $\mathcal{M}_{\kappa}$.
In this case the spectrum possesses a
relatively simple structure as follows. 

\begin{proposition}\label{spektri} 
The spectrum of $\mathcal{M}_{\kappa}$
consists of circles centred at $\kappa$.
\end{proposition}

The eigenvalues of $\mathcal{M}_0$, i.e., circles centred at the
origin, are also called
the coneigenvalues of the matrix $\alp M$ \cite{HJ1}.

To describe methods to compute Krylov subspaces with  $\setR$-linear 
operators, 
we follow \cite[Section 3.1]{EHV}.
Executing the iteration with $\mathcal{M}_{\kappa}$
starting from a vector $b\in \setC^n$,    
we obtain the Krylov subspace
$$\mathcal{K}_j(\mathcal{M}_{\kappa};b)=
{\rm span}\{b,\mathcal{M}_{\kappa}b,\ldots,\mathcal{M}_{\kappa}^{j-1}b\}
$$ $$
= 
{\rm span}\{b,\alp M\overline{b},
\alp M\overline{\alp M}b,\alp M\overline{\alp M}\alp M\overline{b},\ldots\}
$$
which is hence independent of $\kappa$. %
For this an orthonormal basis %
can be computed numerically reliably 
by invoking the  real linear Arnoldi method \cite[p. 820]{EHV}.
In particular, 
if $\dim \mathcal{K}_j(\mathcal{M}_{\kappa};b)=n$ and $Q$ denotes
the respective unitary matrix having the orthonormal basis vectors
as its columns, then $Q^*\alp M \overline{Q}\tau$ is the respective
representation of $\alp M \tau$ in this basis.

The following simple fact is of importance.

\begin{proposition}\label{sfa} 
Let $X\in \setC^{n\times n}$ be invertible.
Then 
$$X^{-1}\mathcal{K}_j(\mathcal{M}_{\kappa};b)=
\mathcal{K}_j(\mathcal{N}_{\kappa};c)$$
with $\alp{N}=X^{-1}\alp{M}\overline{X}$ and $c=X^{-1}b$.
\end{proposition}

If $X$ is unitary, then the corresponding 
sequences of Krylov subspaces are
indistinguishable in the standard Euclidean geometry, i.e., all
the corresponding inner products computed coincide. 
 
In the  $\setR$-linear GMRES
method for solving 
\eqref{sys} suggested in \cite{EHV}, at the $j$th step one imposes the minimum residual condition
$$\min_{z \in 
\mathcal{K}_j(\mathcal{M}_{\kappa};b)}
\left|\left| \mathcal{M}_{\kappa}z-b
\right|\right|$$
for the approximation to satisfy. With appropriate modifications taking into 
account the real linearity, the iteration can be implemented to 
proceed like the classical GMRES \cite{SASU}. 
In particular, if $\alp M$ is either symmetric or skew-symmetric, then the iteration
can be realized in terms of a three term recurrence. %

It is noteworthy that the $\setR$-linear GMRES converges at least as fast
as the standard GMRES applied to the real system of doubled size obtained by 
separating the real and imaginary parts in \eqref{sys}. 
This fact is not surprising. For Krylov subspace methods, {\em not} writing complex
problems in a real form has been advocated already in \cite[p. 446]{FRE}.
Thereby understanding the convergence of the $\setR$-linear GMRES is of central relevance. 

As a final remark, to precondition the linear system \eqref{sys} such that the structure is preserved, see
\cite[Section 4.2]{HP}.

\subsection{Polynomial approximation
problem of the $\setR$-linear GMRES convergence}  
The following notion is needed in what follows. 

\begin{definition} A matrix
$\alp M\in \setC^{n \times n}$ is said to be condiagonalizable  
if there exists an invertible
matrix $X\in \setC^{n \times n}$ such that
\begin{equation}\label{condia}
\alp M=X\alp\Lambda \overline{X^{-1}}
\end{equation}
with a diagonal matrix $\alp \Lambda$.
\end{definition}

The diagonal entries of $\alp \Lambda$ are also called the
coneigenvalues of $\alp M$. (Admittedly, there is a minor,
although trivial inconsistency here compared with the comment 
following
Proposition \ref{spektri}.)

Analytic polynomials are not sufficient to deal with 
real linear operators.  The following subclass
of (polyanalytic) polynomials\footnote{Polyanalytic
polynomials are polynomials in $\lambda$ and $\overline{\lambda}$ \cite{BALK}.} is of central relevance
for the $\setR$-linear GMRES.

\begin{definition} Polynomials of the form
\begin{equation}\label{polap}
\sum_{k=0}^{\lfloor  \frac{j}{2} \rfloor }(\alpha_{2k} +
\alpha_{2k+1}\lambda)\left| \lambda\right|^{2k}
\end{equation}
with $\alpha_k\in \setC$ and, for $j$ even $\alpha_{j+1}=0$,
are denoted by $\mathcal{P}_j(r2)$. Their union 
$\cup_{j=0}^\infty \mathcal{P}_j(r2)$ is 
denoted by $\mathcal{P}(r2).$
\end{definition}

Clearly, $\mathcal{P}_j(r2)$ is a vector space over $\setC$
of dimension $j+1$.
With the restriction $\lambda \in \setR$   
we are dealing with standard analytic polynomials. It is, however, more
natural to contrast $\mathcal{P}_j(r2)$ with radial functions. This and  
the notation used will be explained Section \ref{inles}. 

Observe that problems involving the conjugated variable are 
becoming more common in applications.
Gravitational lensing is one such instance \cite{KHAN}.

Like in the standard GMRES polynomial approximation problem, it is critical how 
well a nonzero constant can be approximated with the elements
of $\mathcal{P}_j(r2)$. As usual, we denote 
the condition number of a matrix $X \in \setC^{n \times n}$ by 
$\kappa_2(X)=
\left|\left| X \right|\right|\left|\left| X^{-1} \right|\right|.$

\begin{theorem}\label{ylara} 
Suppose $\alp M \in \setC^{n \times n}$  is condiagonalizable
as \eqref{condia}.
If $D$ is a  unitary diagonal matrix satisfying
$D^{-1}X^{-1}b\in \setR^n$, then 
$$\min_{z \in 
\mathcal{K}_j(\mathcal{M}_{\kappa};b)}
\left|\left| \mathcal{M}_{\kappa}z-b
\right|\right|
\leq \kappa_2(X)
\min_{p \in \mathcal{P}_{j-1}(r2)}
\max_{\lambda\in \sigma(\alp D)}
\left|
\kappa p(\lambda)+\lambda \overline{p(\lambda)} -1
\right|\left|\left| b
\right|\right|,$$
where $\alp D$ denotes the diagonal matrix $D^{-1}\alp \Lambda\overline{D}$. 
\end{theorem}

\begin{proof} 
Recall that $\mathcal{K}_j(\mathcal{M}_{\kappa};b)=\mathcal{K}_j(\mathcal{M}_{0};b)$ holds for any $\kappa\in \setC$.
Take any $z \in \mathcal{K}_j(\mathcal{M}_{0};b)$
and set $w=D^{-1}X^{-1}z$. Denote $D^{-1}X^{-1}b$ by $r$. 
Since the vector $r$ is real, we obtain
$$w=D^{-1}X^{-1}z =D^{-1}X^{-1} (\sum_{k=0}^{j-1}\alpha_k\mathcal{M}_0^kb)=
\sum_{k=0}^{j-1}\alpha_k\mathcal{D}_0^k (D^{-1}X^{-1}b)=
\sum_{k=0}^{j-1}\alpha_k\mathcal{D}_0^k r$$
$$
= \sum_{k=0}^{\lfloor \frac{j-1}{2}\rfloor}
(\alpha_{2k}+\alpha_{2k+1}\alp D)(\alp D \overline{ \alp D} )^{k}r$$
for some constants $\alpha_k\in \setC$ with
and $\alpha_{2\lfloor (j-1)/2 \rfloor +1}=0$
for $j$ odd. 
This is a polynomial
in a diagonal, i.e., normal
matrix and its adjoint. Hence we have obtained a link between polynomials
in $\lambda$ and $\overline{\lambda}$.
Now we have  
$$\left|\left| \mathcal{M}_{\kappa}z-b
\right|\right|\leq 
\left|\left| XD
\right|\right|
\left|\left| \mathcal{D}_{\kappa}w-r
\right|\right| 
=\left|\left| XD
\right|\right|
\left|\left| \kappa w +\alp D\overline{ w} -r
\right|\right|. 
$$ 
Then  again, since the vector $r$ is real, 
$$
\left|\left| \kappa w +\alp D\overline{ w} -r
\right|\right| \leq
$$$$
\left|\left| 
(\kappa  \sum_{k=0}^{\lfloor \frac{j-1}{2}\rfloor}(\alpha_{2k}+\alpha_{2k+1}\alp D)(\alp D \overline{ \alp D} )^{k}
+\alp D 
\overline{\sum_{k=0}^{\lfloor  \frac{j-1}{2}\rfloor}(\alpha_{2k}+\alpha_{2k+1}\alp D)(\alp D \overline{ \alp D} )^{k}}
   -I)r
\right|\right| 
$$
$$
\leq \max_{\lambda\in \sigma(\alp D)}
\left|
\kappa p(\lambda)+\lambda \overline{p(\lambda)} -1
\right|
\left|\left| (DX)^{-1}
\right|\right|
\left|\left| b
\right|\right|
$$
where $p$ belongs to $\mathcal{P}_{j-1}(r2)$.
Since 
$\left|\left| XD
\right|\right| \left|\left| (XD)^{-1}\right|\right| =\kappa_2(X)$,
the claim follows from this. 
\end{proof}

Observe that in $D^{-1}\alp \Lambda\overline{D}$
the $j$th diagonal entry of $\alp \Lambda$ has been multiplied
by $e^{-2i\theta_j}$, where 
$e^{i\theta_j}$ is the $j$th diagonal entry of $D$.

The key here is the fact that the latter minimisation problem
is of standard type. Being part of classical approximation theory
of functions, there is
no linear algebra involved.\footnote{A term coined by P. Halmos, 
noncommutative approximation theory means matrix (operator)
approximation problems in general.}  
However, unlike the usual GMRES bound
\cite[Section 3.4]{SASU}, 
the point set $\sigma(\alp D)$ 
depends strikingly on the vector $b$. 
(The choice of $D$ to make $X^{-1}b$ real depends on $b$.)
It is a finite 
subset of the spectrum consisting of at most $n$ points, though. 
Generically these points are unique. (Generic here means that
$\alp M$, when condiagonalizable, is assumed to have distinct coneigenvalues.)
Moreover, for an appropriate choice of $b$, 
it can be any subset of the spectrum
with the restriction that
the number coneigenvalues of $\alp M$ of the same modulus does not change. 

The bound shows also that the notion of  ``spectral radius'' for
a diagonalizable antilinear operator is natural. Observe that, by
executing the real linear Arnoldi method, it is straightforward  to
estimate the extreme coneigenvalues of a large (and possibly sparse)
$\alp M$. The rationale is
analogous to the way the classical Arnoldi method yields
eigenvalue approximations.

The convergence behaviour of the CSYM method has been regarded as somewhat puzzling
as well, partly because of the somewhat
unaccesible structure of the appearing Krylov subspaces. 
For some comparisons between other iterative methods, 
see \cite{BUN,KAMA}. (Lack of understanding the convergence
is not just of theoretical interest. It can prevent efficient 
preconditioning.)
The following yields a way to look at it.

\begin{corollary}\label{csymylara} For the CSYM method we 
can choose $X$ to be unitary to have
$$\min_{z \in 
\mathcal{K}_j(\mathcal{M}_{0};b)}
\left|\left| \mathcal{M}_{0}z-b
\right|\right| \leq
\min_{p \in \mathcal{P}_{j-1}(r2)}
\max_{\lambda\in \sigma(\alp D)}
\left|
\lambda \overline{p(\lambda)} -1
\right|\left|\left| b
\right|\right|,$$
where 
 $\alp D=D^{-1}\alp \Lambda\overline{D}$.
\end{corollary}

These bounds are clearly sharp \cite{GRGU}.

Observe that if $\sigma(\alp D)$ is on a line through the origin,
then the CSYM method reduces to the MINRES (minimal residual)
method \cite{PASA} for Hermitian matrices. 
In this case the convergence can be
regarded as well understood. For instance, then  
the convergence can be expected to be faster 
if the origin is not included in the convex hull
of the spectrum. The difference can be dramatic as well.

\section{The probability of %
condiagonalizability}\label{luku3}  
In complex linear matrix analysis, a linear operator is diagonalizable with
probability one. Therefore the analysis of the speed of convergence of
iterations based on classical approximation theory of functions
 on the spectrum  is generically a viable approach. In a typical case it 
can be expected to yield good estimates.
 
Although the set of condiagonalizable matrices includes
complex symmetric matrices, a subspace of $\setC^{n \times n}$
of dimension $n(n+1)/2$, 
assuming condiagonalizability  turns out to be
much more restrictive than assuming diagonalizability. 
Quantitatively this can be expressed
in terms of the following result on random matrices.

\begin{theorem}\label{condiagthm}
Let $\alp M\in\Cnxn$ have entries with real and imaginary
parts drawn independently from the standard normal distribution. Then the
probability that $\alp M$ is condiagonalizable is $2^{-n(n-1)/2}$.
\end{theorem}

One should bear in mind that in practice matrices
possess a lot of structure (such as complex symmetry). 
Thereby, regarding the usage of the bounds of 
Section \ref{luku2} in applications,
this is certainly an overly pessimistic result. 

The rest of this section is dedicated to the proof of Theorem
\ref{condiagthm}.
The probability that a real $n$-by-$n$ matrix with standard normal
entries has only real eigenvalues has been shown to equal $2^{-n(n-1)/4}$
\cite{EDEL}. From Proposition \ref{contriprop} below it is easy to see
that a real matrix is condiagonalizable with the same probability.
For the complex matrices of Theorem \ref{condiagthm},
our computation of the probability proceeds similarly to \cite{EDEL}.

\subsection{Contriangularizable matrices}

We start by recalling basic facts on matrices and consimilarity needed in
the proof. A standard reference here is \cite[Chapter 4]{HJ1}.

\begin{definition} A matrix
$\alp M\in \setC^{n \times n}$ is said to be contriangularizable  
if there exists an invertible
matrix $X\in \setC^{n \times n}$ such that
\begin{equation}\label{contria}
\alp M= X\alp R \overline{X^{-1}}
\end{equation}
with an upper triangular matrix $\alp R$. 
\end{definition}

A matrix $\alp M$ is said to be unitarily contriangularizable if
$\alp M=U\alp R U^T$ with $U$  unitary and $\alp R$ upper triangular.

\begin{proposition}\label{contriprop} Suppose $\alp M \in\Cnxn$. Then
\begin{enumerate}
\item $\alp M$ is contriangularizable if and only if $\alp M$ is unitarily
contriangularizable if and only if all the eigenvalues of $\alp M\overline{\alp M}$ are real and nonnegative.
\item if %
$\alp M=U\alp R U^T$ with $U$  unitary and $\alp R$ upper triangular,
the absolute values of the diagonal entries
of $R$ are always the same, modulo ordering. The
diagonal entries of $\alp R$ can be permuted to any order and chosen
to be real and nonnegative.
\item\label{contridiag}
if $\alp M=U\alp RU^T$ with $U$  unitary and $\alp R$ upper triangular,
where the absolute values $|r_{11}|,|r_{22}|,\dots,|r_{nn}|$ 
of the diagonal entries of $\alp R$ are distinct,
then $\alp M$ is condiagonalizable. Moreover, the set of such matrices
$\alp M$ is open in $\Cnxn$.
\item\label{contriae} The set
\[\set{\alp M\in\Cnxn}{\alp M=U\alp RU^T \text{ with } |r_{ii}|=|r_{jj}| \text{ for some }
i \not= j}\]
is of measure zero. Hence almost all contriangularizable matrices are
condiagonalizable.
\end{enumerate}
\end{proposition}

\begin{proof}
The item (1) is \cite[Theorem 4.6.3]{HJ1}
and the other claims follow readily from the results of
\cite[Section 4.6]{HJ1}. 
\end{proof}

Proposition \ref{contriprop} \eqref{contriae} combined 
with Theorem \ref{condiagthm} yields the corollary that the  probability of a matrix being
contriangularizable is  $2^{-n(n-1)/2}$.

We next prove a uniqueness result which holds true for almost all contriangularizable
matrices. The following lemma is needed.

\begin{lemma}\label{consuurlemma}
Let $R,S\in\Cnxn$ be upper triangular matrices such that
$|r_{ii}| = |s_{ii}|$ and $|r_{ii}|\not=|r_{jj}|$ for all $i\not=j$.
If $U\in\Cnxn$ is a unitary matrix such that
\begin{equation}\label{consuur}
R\overline{U} = US
\end{equation}
then $U$ is a diagonal matrix.
\end{lemma}

\begin{proof}
By Proposition \ref{contriprop} \eqref{contridiag},
$R$ and $S$ are condiagonalizable
and we can find upper triangular invertible matrices $X,Y\in\Cnxn$ such that
\[R=X^{-1}D\overline{X},\qquad S=YD\overline{Y}^{-1},\]
where $D$ is the real diagonal matrix such that $d_{ii}=|r_{ii}|$.
Substituting into \eqref{consuur} we find
\[D\overline{X}\overline{U}\overline{Y} = XUYD.\]
Denoting $E=XUY$, we see that $E$ must be diagonal since $d_{ii}$ are distinct.
Hence $U=X^{-1}EY^{-1}$
is upper triangular and therefore diagonal since $U$ is unitary.
\end{proof}

\begin{proposition}\label{contriuniq}
Let $\alp M\in\Cnxn$ and suppose $\alp M=U\alp RU^T=V \alp SV^T$,
where $U,V$ are unitary, $\alp R,\alp S$ are upper triangular with the same
diagonal consisting of distinct real and positive entries.
Then there exists a diagonal matrix $D\in\Rnxn$ with $\pm 1$ diagonal
entries such that
\begin{equation}\label{contriauniqeq}
\begin{aligned}
U&=VD,\\
\alp R&=D\alp SD.
\end{aligned}
\end{equation}
\end{proposition}

\begin{proof}
From the assumptions we get $\alp R\overline{U^*V} = U^*V\alp S$. By
Lemma \ref{consuurlemma} the matrix $D=U^*V$ is diagonal and we see
that $\alp R=D\alp S D$. Since $\alp R$ and $\alp S$ have
the same nonzero diagonal, the diagonal of $D$ must have $\pm 1$ entries.
\end{proof}

\subsection{Proof of Theorem \ref{condiagthm}}
Since the manipulations that follow require heavily using matrix indices, we
denote the matrix $\alp M$ of Theorem \ref{condiagthm} by $A$.  

The computation of the probability involves evaluating the integral
\begin{equation}\label{pnint}
p_n = \frac{1}{(-4\pi i)^{n^2}}\int_{\mathcal{D}} e^{-\frac{1}{2}\tr(A^*A)}
dA\wedge d\overline{A},
\end{equation}
where $dA = \bigwedge_{i,j=1}^n da_{ij}$ and $\mathcal{D}$ is the set
of condiagonalizable matrices that possess $n$ positive
and distinct coneigenvalues.

To compute $p_n$ we perform the change of variables $A=URU^T$,
where $U$ is unitary, $R\in\mathcal{R}$ and
\[\mathcal{R} = \set{R\in\Cnxn}{R\text{ is upper triangular and }
0 < r_{11} < \cdots < r_{nn}}.\]
To calculate the corresponding Jacobian we
use the notation $[dB]$ to denote the $n\times n$-matrix of the
differential forms $db_{ij}$. Since only the absolute value of the
Jacobian is of interest, in the following we will ignore unconsequential
sign changes due to the anti-commutativity of the wedge product. Also,
we shall ignore the imaginary unit in the volume form, i.e. for 
$z=x+iy$ we write $dz\wedge d\overline{z} = 2\,dx \wedge dy$.
Then
\begin{align*}
[dA] &= [dU]RU^T + U[dR]U^T + UR[dU]^T \\
&=U([dR] + U^*[dU]R + R[dU]^T\overline{U})U^T.
\end{align*}
Denoting
\[[dH] = U^*[dU]\]
we have $[dH]$ skew-Hermitian and therefore
\[[dA] = U[dM]U^T, \quad \text{where} \qquad
[dM] = [dR] + [dH]R - R[d\overline{H}].\]
Hence
\[dA = \det(U)^{2n}dM\]
and
\begin{equation}\label{daeqdm}
dA \wedge d\overline{A} = dM \wedge d\overline{M}.
\end{equation}
We now divide the calculation to three cases
\begin{equation}\label{dmdconjm}
dM \wedge d\overline{M} =
\bigwedge_{i > j} (dm_{ij} \wedge d\overline{m_{ij}}) \wedge
\bigwedge_{i} (dm_{ii} \wedge d\overline{m_{ii}}) \wedge
\bigwedge_{i < j} (dm_{ij} \wedge d\overline{m_{ij}}).
\end{equation}
Suppose first that $i > j$. Then
\begin{equation}\label{dmij1}
dm_{ij} = dh_{ij}r_{jj} - r_{ii}d\overline{h_{ij}}
+ \sum_{k < j} dh_{ik}r_{kj} - \sum_{k > i}r_{ik}d\overline{h_{kj}}.
\end{equation}
Actually
\begin{equation}\label{dmij2}
\bigwedge_{i > j} (dm_{ij} \wedge d\overline{m_{ij}}) =
\bigwedge_{i > j} (r_{jj}^2 - r_{ii}^2)dh_{ij}\wedge d\overline{h_{ij}}.
\end{equation}
To see this, first note that
\begin{equation}
(dh_{ij}r_{jj} - r_{ii}d\overline{h_{ij}}) \wedge
(d\overline{h_{ij}}r_{jj} - r_{ii}dh_{ij})
= (r_{jj}^2 - r_{ii}^2)dh_{ij}\wedge d\overline{h_{ij}}.
\end{equation}
That the last two summations in \eqref{dmij1} make no contribution
to \eqref{dmij2}, consider ordering their terms first by the increasing
second index $v$ of $dh_{uv}$ (and $d\overline{h_{uv}}$) and then by the
decreasing first index $u$. The elimination starts with $dh_{n1}$ (and
$d\overline{h_{n1}}$) and proceeds in the described order. We repeatedly
use the reduction
\begin{align*}
\bigwedge_{i > j} (dm_{ij} \wedge d\overline{m_{ij}}) &=
\omega_1 \wedge (\omega_2 + \gamma\, dh_{uv}) \wedge
(r_{vv}^2 - r_{uu}^2)dh_{uv}\wedge d\overline{h_{uv}} \\ 
&= \omega_1 \wedge \omega_2 \wedge
(r_{vv}^2 - r_{uu}^2)dh_{uv}\wedge d\overline{h_{uv}},
\end{align*}
where $\omega_1,\omega_2$ are some differential forms and $\gamma$ is
$\pm$ some entry of $R$.

We next consider the case $i=j$ in \eqref{dmdconjm}. Now
\begin{equation}\label{dmii1}
dm_{ii} = dr_{ii} + 2r_{ii}dh_{ii}
+ \sum_{k < i} dh_{ik}r_{ki} - \sum_{k > i} r_{ik}d\overline{h_{ki}}.
\end{equation}
We get
\begin{equation}\label{dmii2}
\begin{aligned}
\bigwedge_{i > j} (dm_{ij} \wedge d\overline{m_{ij}}) &\wedge
\bigwedge_{i} (dm_{ii} \wedge d\overline{m_{ii}}) \\ &=
\bigwedge_{i > j} (dm_{ij} \wedge d\overline{m_{ij}})
\wedge \bigwedge_i (4r_{ii} dh_{ii} \wedge dr_{ii}),
\end{aligned}
\end{equation}
since the terms in the last two summations in \eqref{dmii1} are eliminated
due to \eqref{dmij2}.

The remaining case is $i < j$. Now
\begin{equation}\label{dmij10}
dm_{ij} = dr_{ij} + \sum_{k < j} dh_{ik}r_{kj} -
\sum_{k > i} r_{ik}d\overline{h_{kj}}.
\end{equation}
All terms in the last two summations are now eliminated due to
\eqref{dmii2} so that we finally get
\begin{equation}\label{dmvolform}
\begin{aligned}
dM \wedge d\overline{M} &=
4^n \prod_i r_{ii} \prod_{i < j} (r_{jj}^2 - r_{ii}^2)
\bigwedge_{i < j} (dh_{ij} \wedge d\overline{h_{ij}} \wedge
dr_{ij} \wedge d\overline{r_{ij}}) \wedge \\ &
\bigwedge_i (dh_{ii} \wedge dr_{ii}).
\end{aligned}
\end{equation}
We then use \eqref{daeqdm} to compute the integral \eqref{pnint}
by integrating over the unitary group and the upper triangular matrices $R$
\[p_n = \frac{1}{2^n(4\pi)^{n^2}} \int_{U(n)\times \mathcal{R}}
e^{-\frac{1}{2}\tr(R^*R)}\,dM\wedge d\overline{M},\]
where the factor $2^n$ corresponds to the fact that by Proposition
\ref{contriuniq} integration over $U(n)\times \mathcal{R}$ counts all
matrices $A$ precisely $2^n$ times.

The volume of the unitary group \cite[Proposition 4.1.14]{AGZ} is
\[\int \bigwedge_{i < j} (dh_{ij} \wedge d\overline{h_{ij}}) \wedge
\bigwedge_i dh_{ii} = \prod_{j=1}^n \frac{(2\pi)^j}{(j-1)!}.
\]
The integral over the strict upper triangular part of $R$ is
\[\int e^{-\frac{1}{2}\sum_{i < j} |r_{ij}|^2} \bigwedge_{i < j}
(dr_{ij} \wedge d\overline{r_{ij}}) = (4\pi)^{n(n-1)/2}.\]
The integral over the diagonal of the matrices $R$ can be computed using
Selberg's integral \cite[Formula 17.6.6]{MEH}
\begin{align*}
&\int_{\mathcal{\diag(R)}}
e^{-\frac{1}{2}\sum_i {r_{ii}^2}}
\prod_i r_{ii} \prod_{i < j} (r_{jj}^2 - r_{ii}^2)\,
dr_{11} \cdots dr_{nn} \\
&= \frac{1}{n!} \int_0^\infty \cdots \int_0^\infty
e^{-\frac{1}{2}\sum_i {r_{ii}^2}}
\prod_i r_{ii} \prod_{i < j} |r_{jj}^2 - r_{ii}^2|\,
dr_{11} \cdots dr_{nn} = \prod_{j=1}^{n-1} j!
\end{align*}
Hence
\begin{equation*}
p_n = \frac{1}{2^n(4\pi)^{n^2}}
4^n (2\pi)^{n(n+1)/2} (4\pi)^{n(n-1)/2}
= 2^{-n(n-1)/2}.
\end{equation*}

\section{Complex symmetry, orthogonal polynomials 
and three term recurrence}\label{luku4}  
The connection between the Hermitian Lanczos method, 
Hermitian Jacobi, i.e., Hermitian tridiagonal matrices  and
orthogonal polynomials is standard material in numerical linear algebra
and classical analysis;
see, e.g.,  \cite{Gau, GME}
and \cite{SZE,SIBOOK}.

Condiagonalizability is a special property which implies that a linear
algebra problem turns into a problem in classical approximation theory. 
In what follows, an analogous connection
for antilinear operators involving a complex symmetric matrix is described. 
For complex symmetric matrices, see 
the classical publications listed in \cite[p. 218]{HJ1}.
See also \cite{GP} and references therein
for complex symmetric operators on separable Hilbert spaces.

\subsection{Construction}\label{lequ}  
For the connection, consider an antilinear operator
$$%
\alp M\tau$$
on $\setC^n$ involving a complex symmetric matrix $\alp M$.  
(We could equally well consider $\mathcal{M}_{\kappa}$
but for the simplicity of the presentation, 
we set $\kappa=0$.)
Take a unit vector $b\in \setC^{n}$.
Then executing 
the real linear Arnoldi method yields us a tridiagonal complex symmetric
matrix, i.e., a complex symmetric Jacobi matrix 
because of the following fact.

\begin{proposition}\cite{EHV} If $\alp{M}^T=c \alp M$ with $c =\pm 1$,
then the real linear Arnoldi method is realizable with a three term recurrence.
\end{proposition}

Because of the way the real linear Arnoldi method proceeds,
in the resulting tridiagonal complex symmetric matrix 
there can appear complex entries only on the diagonal. In
what follows, when
$c=1$, the real linear Arnoldi method
is called the complex symmetric Lanczos method.

As in the proof of Theorem
\ref{ylara}, choose a unitary matrix $U$ such that
\begin{equation}\label{ehto}
\alp D=U^*\alp M\overline{U}\, \mbox{ is diagonal and }\, r=U^*b\in \setR^n
\end{equation}
holds. Then $\mathcal{K}_j(\alp M\tau;b)$ is unitarily equivalent to
$\mathcal{K}_j(\alp D\tau;r)$ in the sense of 
Proposition \ref{sfa}. For the latter Krylov subspace, the conjugations
affect $\alp D$ only, yielding polynomials in
$\alp D$ and $\overline{\alp D}$ which correspond to elements
of $\mathcal{P}_j(2r)$ in a natural way. 

We assume that for any triple of the nonzero coneigenvalues of $\alp M$,    
at most two of them can share the same modulus, and, if zero is 
a coneigenvalue, it appears just once. 
This assumption holds generically.
Moreover, we assume the starting vector $b \in \setC^n$ to be generic
in the sense that the eigenvalues of $\alp D$ are distinct
and all the entries of $r$ are strictly positive. 

By Proposition \ref{sfa} (and the comment
that follows), the Jacobi matrix computed by the 
complex symmetric Lanczos method  with
$\alp M\tau$ using the starting vector $b$ yields the same 
Jacobi matrix as when executed with $\alp D\tau$ using the
starting vector $r$. 
(Of course, the orthonormal bases generated differ
according to Proposition \ref{sfa}.)
Denote the entries of this matrix as 
\begin{equation}\label{jtri}
\alp J=\left[ \begin{array}{ccccc} 
\alpha_1&\beta_1&0&\cdots&0\\
\beta_1&\alpha_2&\ddots&\cdots&0 \\
\vdots&\ddots&\ddots& \ddots&\vdots \\
0&\ldots&&\alpha_{n-1}&\beta_{n-1}\\
0&\ldots&0&\beta_{n-1}&\alpha_n
\end{array} \right],
\end{equation}
so that the corresponding antilinear operator is $\alp J \tau$.
The complex symmetric Lanczos method is devised in such a way that
the entries satisfy $\alpha_j\in \setC$ and $\beta_j>0$
assuming the method does not break down. (When the  
classical Hermitian Lanczos method is executed, 
the respective entries 
satisfy $\alpha_j\in \setR$ and $\beta_j>0$.)

For the converse, assume given $\alp J$ and the task is construct 
a diagonal matrix $ D$ and a real vector $v$
giving $\alp J$ after executing the complex symmetric
Lanczos method. This can be accomplished
by computing a unitary matrix $V$ whose first column $v$ is real
such that $\alp J\tau =V^* D\overline{V}\tau$ with a diagonal matrix $ D$.

We have lack of uniqueness in the case there appears two
coneigenvalues of the same modulus. In \eqref{ehto} this takes place
since for any isometric $V\in\C^{n\times 2}$ we have $VV^T = (VR)(VR)^T$
for all orthogonal matrices $R\in\R^{2\times 2}$. For the sake of completeness,
the following proposition contains the converse.

\begin{proposition} 
Suppose
$UU^T=VV^T$ for two isometric matrices $U,V\in \setC^{n \times m}$.
Then $V=UR$ for a unitary matrix $R\in \setR^{m \times m}$. 
\end{proposition}
\begin{proof}
We have $U^*V = U^*VV^T\overline{V} = U^*UU^T\overline{V} = U^T\overline{V}
=\overline{U^*V}$. Take $R=U^*V$.
\end{proof}

For orthogonal polynomials, associate with each point $\lambda_j\in \sigma(\alp D)$ 
the weight $r_j^{2}$, where $r_j$ is the
$j$th entry of the vector $r$. Denote by $\inprod{\cdot,\cdot}$ the standard 
Euclidean inner product
on $\setC^n$. Then an inner product on 
$\mathcal{P}_j(r2)$
corresponding to the complex symmetric Lanczos method is  
defined as 
$$\inprod{p,q}=\inprod{p(\alp M\tau)b,q(\alp M\tau)b} $$$$
=\inprod{p(\alp D\tau)r,q(\alp D\tau)r}=
\sum_{k=1}^np(\lambda_k)\overline{q(\lambda_k)}r_k^2,$$ 
where we used 
$p(\alp M\tau)b= 
\sum_{k=0}^{j}\alpha_k(\alp M\tau)^kb=
U\sum_{k=0}^{j}\alpha_k(\alp D\tau )^k r$$
= U\sum_{k=0}^{\lfloor \frac{j}{2}\rfloor}
(\alpha_{2k}+\alpha_{2k+1}\alp D)(\alp D \overline{ \alp D} )^{k}r$
and similarly for $q(\alp M\tau)b$.

Consider the (discrete) monomial functions in \eqref{mono}.
In terms of the Jacobi matrix entries in \eqref{jtri},   
the three term recurrence
for computing the  respective orthogonal polynomials
can be expressed as   
\begin{equation}\label{orhopol}
 \begin{array}{rll} 
p_0(\lambda)&=&1\\
\beta_1p_1(\lambda)&=&\lambda \overline{p_0(\lambda)}-\alpha_1p_0(\lambda) \\
\beta_2p_2(\lambda)&=&\lambda \overline{p_1(\lambda)}
-\alpha_2p_1(\lambda)-\beta_1p_0(\lambda) \\
\beta_3p_3(\lambda)&=&\lambda \overline{p_2(\lambda)}
-\alpha_3p_2(\lambda)-\beta_2p_1(\lambda) 
\end{array} 
\end{equation}
and so on. Note that the assumption of $\alp D$ having distinct eigenvalues
with any triple of them having at most two values of the same modules
together with $r_j^2 >0$ for all $j$ and
Proposition \ref{zeroprop} below implies that the
Lanczos method does not break down.

A number $\lambda\in\C$ is called a zero of $p\in\mathcal{P}(r2)$ if
$p(\lambda)=0$. 
\begin{proposition}\label{zeroprop}
Let $p \in \mathcal{P}_j(r2)$ be nonzero. The following claims hold:
\begin{enumerate}
\item\label{zerit1} If $p$ has two distinct zeroes of the same modulus,
then all numbers of that modulus are zeroes.
\item Let $m$ be the number of moduli for which all numbers of that modulus are
zeroes and let $s$ be the number of moduli for which exactly one number
is a zero. Then $2m+s \leq j$.
\end{enumerate}
\end{proposition}

\begin{proof}
Let $u$ and $v$ be (ordinary) polynomials of degrees
at most $\lfloor \frac{j}{2} \rfloor$
and $\lfloor \frac{j-1}{2} \rfloor$, respectively, such that
\begin{equation}\label{zereq1}
p(\lambda) = u(|\lambda|^2) + \lambda v(|\lambda|^2).
\end{equation}
By the assumption of Item \ref{zerit1}, there exist $\lambda_1$ and
$\lambda_2$ such that $\lambda_1 \not= \lambda_2,\,|\lambda_1|=|\lambda_2|$
and $p(\lambda_1) = p(\lambda_2) = 0$. This together with \eqref{zereq1}
implies $u(|\lambda_1|^2)=v(|\lambda_1|^2)=0$ proving the first claim.

Let $M_1,\dots,M_m$ be the moduli for which all numbers of these moduli
are zeroes.
By factoring, there exist (ordinary) polynomials
$\widetilde{u}$ and $\widetilde{v}$ such that
\begin{align*}
u(|\lambda|^2) &= \widetilde{u}(|\lambda|^2)
\prod_{i=1}^m (|\lambda|^2-M_i^2), \quad
v(|\lambda|^2) = \widetilde{v}(|\lambda|^2)
\prod_{i=1}^m (|\lambda|^2-M_i^2).
\end{align*}
Note that $\deg(\widetilde{u})\leq\lfloor \frac{j}{2} \rfloor - m$ and
$\deg(\widetilde{v}) \leq \lfloor \frac{j-1}{2} \rfloor - m$. Let $\lambda$
be a zero of $p$ such that no other number is a zero of the same modulus.
Then $\widetilde{u}(|\lambda|^2) + \lambda\widetilde{v}(|\lambda|^2) = 0$ and
\[\widetilde{u}(|\lambda|^2)\overline{\widetilde{u}(|\lambda|^2)} -
|\lambda|^2\widetilde{v}(|\lambda|^2)\overline{\widetilde{v}(|\lambda|^2)}
= 0,\]
where the left-hand side is a nonzero (ordinary) polynomial in $|\lambda|^2$ of
degree at most $j - 2m$. Hence $s\leq j - 2m$.
\end{proof}

Note that $p$ need not have any zeroes at all.

If $\sigma(\alp D)\subset \setR$ holds, then
the conjugations are vacuous and we have the classical symmetric 
Lanczos method
 \cite{PA}.\footnote{By the (classical) symmetric 
Lanczos method we mean the three term recurrence
for transforming a real symmetric matrix
into tridiagonal form.}
And conversely, if 
$\sigma(\alp D)\not\subset \setR$ holds, then
we have a natural extension of the symmetric Lanczos method
preserving the length of recurrence. Thereby, the numerical behaviour 
in finite precision, 
i.e., the loss of orthogonality among vectors computed can be expected 
to be similar to the classical symmetric Lanczos method. See 
\cite[Chapter 13.3]{PA} for the effects of finite precision then.

Certainly, complex symmetric Jacobi matrices  
can be treated in the $\setC$-linear setting \cite{BEC}. 
(Then one has to deal with formal orthogonal polynomials.)
However, we do not find it perhaps quite 
as natural as through the connection with the complex symmetric Lanczos 
method prescribed.

\subsection{Interpolation and least squares approximation} \label{inles} 
For the interpolation with the elements of $\mathcal{P}_{j}(r2)$, it is straightforward to
construct  Vandermonde-type matrices from
the monomials \eqref{mono}. (Numerically this is not advisable, though.) 
To understand their invertibility,
consider the case  
of having exactly two interpolation
points for each appearing modulus. 
That is, assume
there are $k$ different moduli $r_1>r_2 >\cdots >r_k$ and $2k$ points in all. Take the Lagrange
interpolation basis polynomials 
$$l_l(|\lambda|^2)=
\prod_{1 \leq m \leq k,\, m \not=l} 
\frac{|\lambda|^2-r_l^2}{r_l^2-r_m^2}.$$
Hence, $l_l(|\lambda_l|^2)=1$ while $l_l(|\lambda_m|^2)=0$ for $m\not=l$.
Now, for any two distinct interpolation nodes with  modulus $r_l$, take
the unique interpolating polynomial $p_2(\lambda)=c_l+d_l\lambda$. Then 
$p_2l_l\in \mathcal{P}_{j}(r2)$ with $j=2k-1$. Taking the sum of these
yields the required interpolant.

Consequently, the notation  $\mathcal{P}(r2)$ used is explained as follows. With the
elements of $\mathcal{P}_{j}(r2)$ we may interpolate at most two points on
a circle. Recall that with the radial polynomials
$\Sigma_{k=0}^ja_k|\lambda|^k$ one can interpolate
at most one point on a circle. 
Repeating this idea, it is clear how to define  $\mathcal{P}_{j}(rk)$
in such a way that we may interpolate at most $k$ points on
a circle. Hence we have a natural extension of radial functions. 

Since numerically computations involving orthogonal functions is 
preferable, interpolation with the elements of $\mathcal{P}_{j}(r2)$
should be performed by executing the complex symmetric Lanczos method 
just described. (Of course, for the classical symmetric 
Lanczos method this is a standard approach already from the
late 1950s \cite{FOR,Gau}.)
This is straighforward by choosing $\alp D$ with the
diagonal entries equaling the 
interpolation nodes and $r$ any unit vector supported at the nodes.

\subsection{Approximation of continuous functions on curves}
The interpolation scheme just presented suggests on what kind of
curves we can expect the approximation to be successful. 
For approximating continuous functions with the elements from
$\mathcal{P}(r2)$ we now give a generalization of the Weierstrass
approximation theorem. We start with the following lemma.

\begin{lemma}\label{jordanarclemma}
Let $\gamma \subset \C$ be a compact simple open curve such that
$\gamma$ intersects every origin centred circle in at most two points.
Then $\gamma$ can be extended to a simple closed curve $\widetilde{\gamma}
\subset \C$ such that $\gamma \subset \widetilde{\gamma}$ and
$\widetilde{\gamma}$ intersects every origin centred circle in at most
two points.
\end{lemma}

\begin{proof}
Let $r_1 \geq 0\,$ (and $r_2 \geq 0$) be the supremum (infimum) of the
values $r$ such
that every origin centred circle with radius at most (at least) $r$ does not
intersect $\gamma$ (if $0\in\gamma$ then define $r_1=0$).
Then the circle of radius $r_1$
intersects $\gamma$ either in one point or two points. Similarly
for the circle of radius $r_2$. In case of two intersection points,
it is easy to extend $\gamma$ so that without loss of generality we
may assume the circles of radius $r_1$ and $r_2$ each intersect $\gamma$
at exactly one point.

Let $\rho_1\,$ (and $\rho_2$) be the supremum (infimum) of the values $r$
such that every origin centred circle with radius $\rho$, where
$r_1 < \rho < r\,\,(r < \rho < r_2)$, intersects $\gamma$ in exactly two points
(we may assume such values $r$ exist since $\gamma$ can be easily extended
to accommodate this).
It follows that circles of radius $r$ such that $\rho_1 < r < \rho_2$
intersect $\gamma$ at exactly one point which we denote $z(r)$.
Denote by $w_j$ and $v_j\,\,(j=1,2)$
the intersection points of the arc $\gamma$ with the circle of radius $\rho_j$
and further choose $w_j$ as one of the end points of the arc $\gamma$. Note
that $|w_j|=|v_j|=\rho_j$ and by defining $z(\rho_j)=v_j$ the function
$r \mapsto z(r)$ becomes continuous in $[\rho_1,\rho_2]$.

Let
\[\epsilon = \frac{1}{2}\min(|v_2\frac{w_1}{v_1} - v_2|,|w_2 - v_2|)\]
and (by continuity) choose $R$ such that $\rho_1 < R < \rho_2$ and
\begin{equation}\label{jarccont}
|z(r) - v_2| < \epsilon \qquad \text{for all } r > R.
\end{equation}
We define the extension
$\widetilde{\gamma}$ as follows.
Let
\[\gamma_0 = \set{z(r)\frac{w_1}{v_1}}{ \rho_1 \leq r \leq R}.\]
Note that this is a rigid rotation and therefore the simple curve
$\gamma \cup \gamma_0$ intersects all circles in at most two points.
Next, let $\alpha = z(R)w_1/v_1$ and
\[\gamma_1 = \set{r \exp\left(i\arg(\alpha)\frac{\rho_2-r}{\rho_2-R} + i\arg(w_2)\frac{r-R}{\rho_2-R}\right)}{R \leq r \leq \rho_2}.\]
Note that $\widetilde{\gamma} = \gamma \cup \gamma_0 \cup \gamma_1$ is closed
and intersects
all circles in at most two points. Due to \eqref{jarccont} it is also a
simple curve provided the direction of rotation of the spiral is properly
chosen either
clockwise or counter-clockwise, i.e. we choose a real number $t$ such
that $t \leq \arg(\alpha),\arg(w_2) < t+2\pi$.
\end{proof}

\begin{theorem}\label{weier}
Let $\gamma \subset \C$ be a compact simple (open or closed) curve such that
$\gamma$ intersects every origin centred circle in at most two points.
Let $f:\gamma \to \C$ be a continuous function and suppose $\epsilon > 0$.
Then there exists a polynomial $p\in\mathcal{P}(r2)$ such
that
\begin{equation}\label{weierdiff}
\max_{z\in\gamma} | f(z) - p(z) | < \epsilon.
\end{equation}
\end{theorem}

\begin{proof}
By Lemma \ref{jordanarclemma} we may assume $\gamma$ is a closed simple curve.
Let $r_1 \geq 0\,$ (and $r_2 \geq 0$) be the supremum (infimum) of the
values $r$ such
that every origin centred circle with radius at most (at least) $r$ does not
intersect $\gamma$ (if $0\in\gamma$ then define $r_1=0$).
Then the circle of radius $r_j\,\,(j=1,2)$
intersects $\gamma$ at exactly one point which we denote by $w_j$.
Furthermore,
every circle of radius $r$ such that $r_1 < r < r_2$ intersects $\gamma$
at exactly two points which we denote by $z_1(r)$ and $z_2(r)$ chosen
in one of the two ways to make $r \mapsto z_j(r)$ continuous $(j=1,2)$.
We also define $z_1(r_1)=z_2(r_1)=w_1$ and $z_1(r_2)=z_2(r_2)=w_2$.

It is easy to see that there exists a continuous function $g:\gamma \to \C$
such that $g$ is constant in a neighbourhood of $w_1$ and a neighbourhood of
$w_2$ and
\[\max_{z\in\gamma} | f(z) - g(z) | < \frac{\epsilon}{2}.\]
We then define the functions $a_1,a_2:[r_1,r_2]\to\C$ by
\begin{equation}\label{weierab}
\begin{aligned}
a_1(r) &= g(z_1(r)) - z_1(r)\frac{g(z_2(r)) - g(z_1(r))}{z_2(r)-z_1(r)}, \\
a_2(r) &= \frac{g(z_2(r)) - g(z_1(r))}{z_2(r)-z_1(r)}.
\end{aligned}
\end{equation}
The functions $a_1$ and $a_2$ are continuous since we chose $g$ to be
constant near $w_1$ and $w_2$.
Note that $g(z)=a_1(|z|) + a_2(|z|)z$ for all $z\in\gamma$. By the Weierstrass
approximation theorem for compact intervals on the real line, there exists
ordinary polynomials $p_1$ and $p_2$ such that
\[\max_{r_1\leq r \leq r_2} |a_j(r) - p_j(r^2)| < \frac{\epsilon}{4(1 + r_2)}\qquad
(j=1,2).\]
We then define $p(z) = p_1(|z|^2) + p_2(|z|^2)z$ and see that $p\in
\mathcal{P}(r2)$. Also
\[|g(z) - p(z)| \leq |a_1(|z|) - p_1(|z|^2)| + |a_2(|z|) - p_2(|z|^2)||z|
< \frac{\epsilon}{2}\]
for all $z\in\gamma$. The estimate \eqref{weierdiff} then readily follows.
\end{proof}

It is noteworthy that compact subsets of $\setR$ are admissible. This is the
case in the Hermitian Lanczos method.

The exponential function is the most important example of a nonrational 
(certainly continuous) function. In the present context we obtain
it as a limit of elements in $\mathcal{P}(r2)$ as follows.

\smallskip

\begin{example} Consider a condiagonalizable $\alp M \in \setC^{n\times n}$
as in \eqref{condia}.
The corresponding semigroup is defined as 
$$
e^{t\alp M \tau}=
\sum_{j=0}^{\infty} \frac{(t\alp M\tau)^j}{j!}
$$
for $t\in \setR$. (Then $e^{t\alp M \tau}x_0$ solves
the initial value problem $x'=\alp M \overline{x}$, $x(0)=x_0$.) 
When applied to a vector $b\in \setC^n$ such that $D^{-1}X^{-1}b=r\in \setR^n$, 
we obtain the associated exponential function
\begin{equation}\label{expo}
\sum_{j=0}^{\infty} (\frac{1}{(2j)!}+\frac{\lambda}{(2j+1)!})\left|\lambda \right|^{2j}.
\end{equation}
when looking at the problem in the corresponding basis.
Of course, this reduces to the standard exponential function
for $\lambda \in \setR$.
\end{example}

\smallskip

\section{Numerical experiments}\label{luku5}

We now present very preliminary numerical experiments 
on the relationship between $\mathcal{P}_j(2r)$ and
the convergence of the $\R$-linear GMRES method. For simplicity, we focus 
specifically on the CSYM method. 
We do not have solutions to the polynomial
minimization problems of Theorem \ref{ylara} and Corollary
\ref{csymylara}. However, numerical results on the respective
diagonal linear systems  with a real right-hand side
unveil some of the intricacies
of the latter problem.

In all the examples given below the CSYM method is thus executed to solve
\[\alp D \overline{x} = r,\]
where $\alp D \in \Cnxn$ is a diagonal matrix and
$r\in\R^n$ is such that all its entries are ones. 
The diagonal entries of $\alp D$ are set as 
\begin{equation}\label{numdiag}
d_{jj} = R_j e^{2\pi i\phi_j},
\end{equation}
where $R_1 = 1$ and $R_n = 10$ while the other values of $R_j$ are linearly interpolated
between these two extremes. The angles $\phi_j$ are specificied in each
case separately and described below. For each example we plot
the diagonal of
$\alp D$ and the $\log_{10}$ of the relative residual
$\norm{r - \alp D \overline{x_j}}/\norm{r}$, where
$x_j \in \mathcal{K}_{j-1}(\alp D \tau;r)$ is the minimizing vector
with the starting vector $x_0=0$. We used $n=500$ in each problem.

Before describing the examples, we want to mention
a feature which we find puzzling.
Namely, the numerical results depend on how accurately the entries 
\eqref{numdiag} are generated. This
is illustrated in the first example below. All the
computations
were carried out in {\tt MATLAB}\footnote{Version 7.10.0.499 (R2010a)}
variable-precision arithmetic with an accuracy
of 30 decimals (recall that double-precision floating point numbers
have approximately 16 decimals). The high-precision arithmetic was chosen
in order to show that the apparent numerical instability
of double-precision floating
point computations 
seems to result from
the input $\alp D$ itself rather
than any serious cancellation effect in the minimal residual algorithm.
There is no qualitative change to the results by using more than 30
decimals of accuracy. At the moment we do not have an explanation
for this behaviour.

The actual examples are set up as follows. In the first
example we illustrate the comment made
after Corollary \ref{csymylara}, i.e., when $\sigma (\alp D)$
is on a line through the origin, 
the CSYM method reduces to the MINRES method.
Then in the examples that follow, the line is deformed into
more complicated shapes. The rate of convergence slows down accordingly.

\begin{itemize}
\item {\bf Example 1}.
Here we chose $\phi_j = 1/10$ for all $j$; see 
the left panel of Figure \ref{constantphifig}  for $\sigma (\alp D)$.
The matrix $\alp D$ was first computed in high-precision and in this case
the residual dropped in a straight line. The matrix $\alp D$ was then
converted
to double-precision format and back to high-precision format. The
computation was performed again giving the slower convergence starting
at approximately $10^{-6}$. Similar effect would be seen in Example 2
as well if the precision of the input $\alp D$ was lowered.

\begin{figure}
\begin{center}
\includegraphics{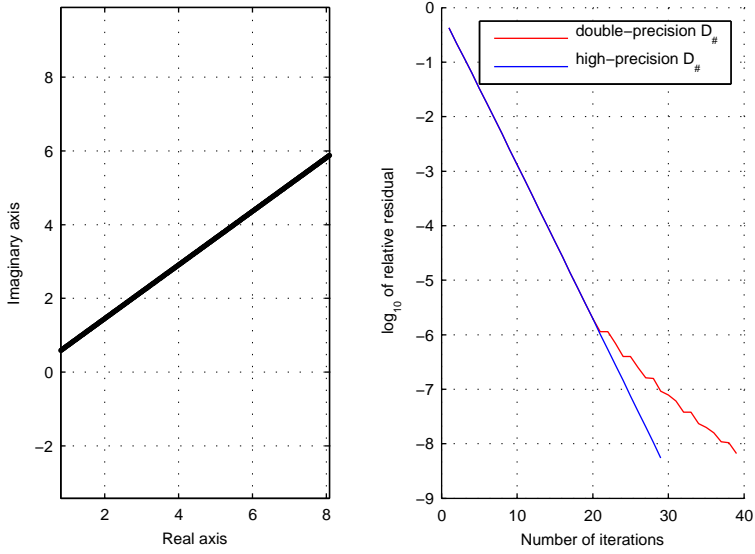}
\caption{Example 1}
\label{constantphifig}
\end{center}
\end{figure}

\begin{figure}
\begin{center}
\includegraphics{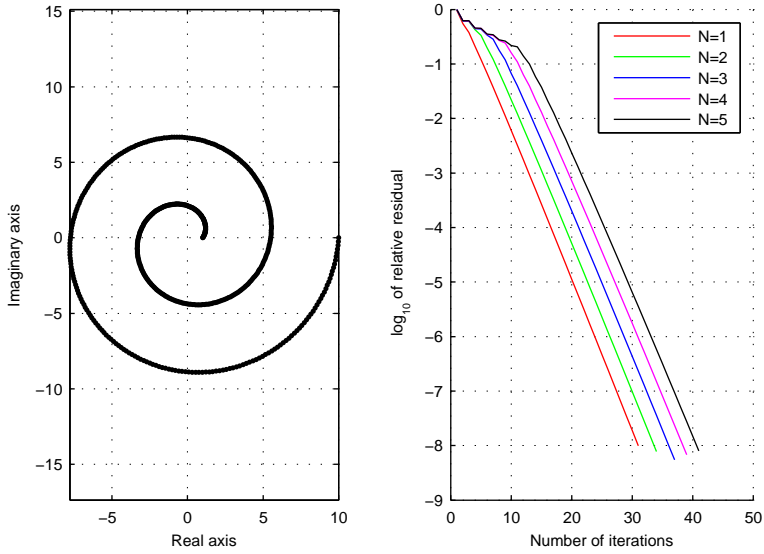}
\caption{Example 2. The diagonal of $\alp D$ in the case $N=2$ is plotted
on the left panel.}
\label{manyspiralsfig}
\end{center}
\end{figure}

\item {\bf Example 2}.
Here we computed using five different sets of angles. We chose
$\phi_1^{(N)} = 0, \phi_n^{(N)} = N$, where $N=1,\dots,5$, and
the rest of $\phi_j^{(N)}$ were linearly interpolated between the extremes.
See Figure \ref{manyspiralsfig}.

\begin{figure}
\begin{center}
\includegraphics{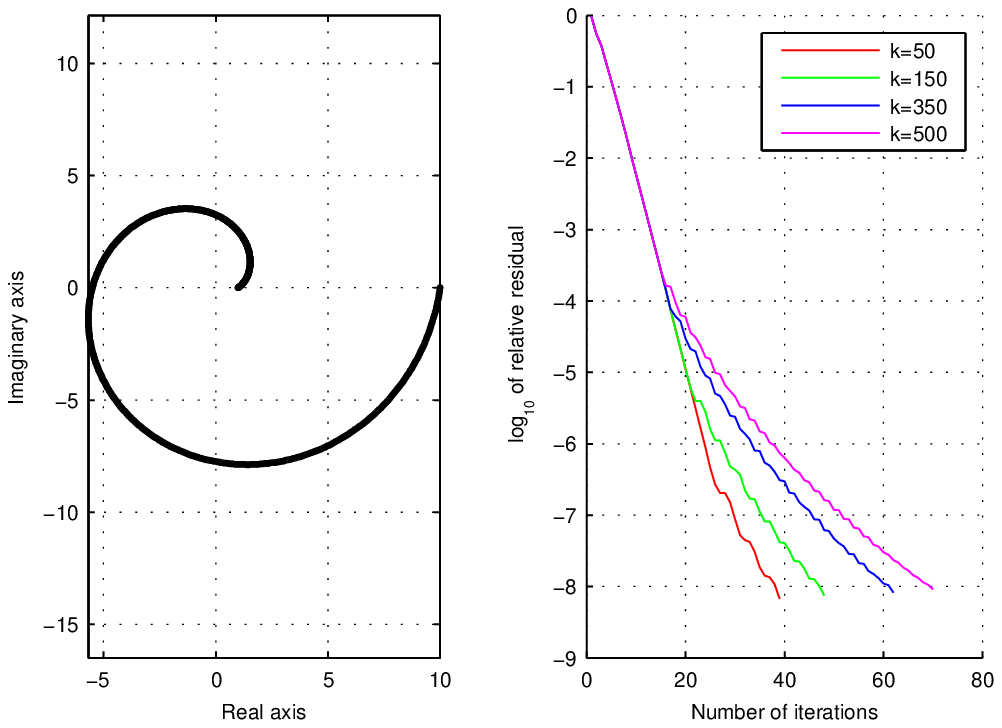}
\caption{Example 3}
\label{innerrandfig}
\end{center}
\end{figure}

\begin{figure}
\begin{center}
\includegraphics{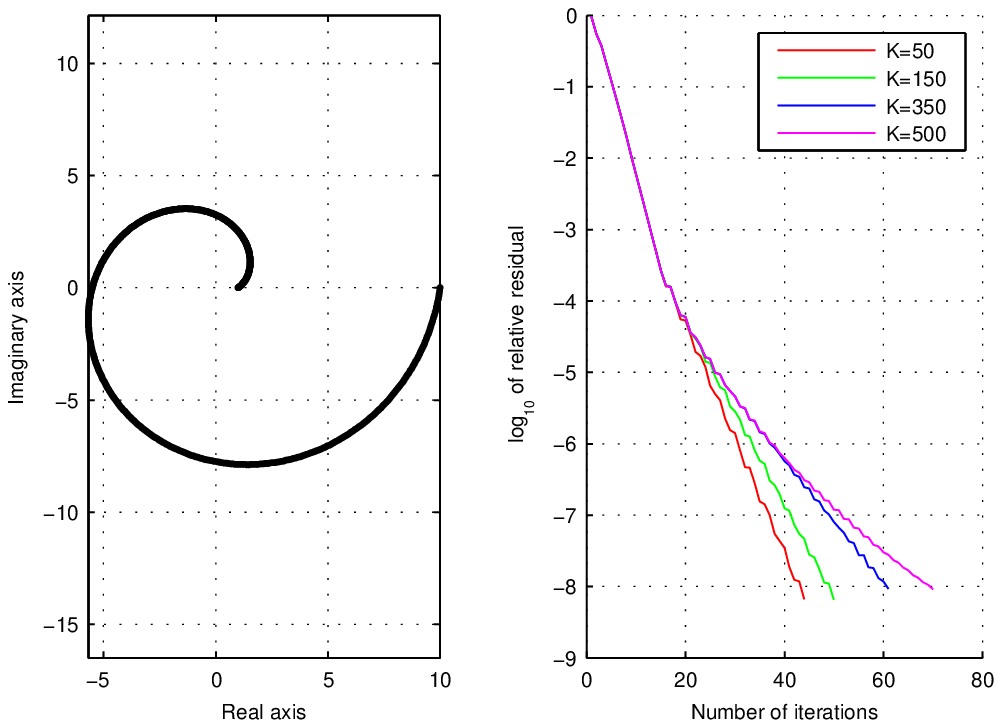}
\caption{Example 4}
\label{outerrandfig}
\end{center}
\end{figure}

\item {\bf Example 3}.
Let $\widetilde{\phi}_1=0, \widetilde{\phi}_n=1$ and the rest of
$\widetilde{\phi}_j$
linearly interpolated between the extremes (Example 2 case $N=1$).
Additionally, a random vector $\rho\in\C^n$ was generated with
entries uniformly distributed between $0$ and $10^{-10}$.
Given an integer $k$ such that $1\leq k \leq n$, we chose
four different sets of angles in \eqref{numdiag} by
$\phi_j^{(k)} = \widetilde{\phi}_j + \rho_j$ for $j \leq k$
and $\phi_j^{(k)} = \widetilde{\phi}_j$ for $j > k$.
The results are displayed in Figure \ref{innerrandfig}.

\item {\bf Example 4}.
Let $\widetilde{\phi}_j$ and $\rho$ be as in Example 3.
Given an integer $K$ such that $1\leq K \leq n$, we now chose
four different sets of angles in \eqref{numdiag} by
$\phi_j^{(K)} = \widetilde{\phi}_j$ for $j \leq n-K$
and $\phi_j^{(k)} = \widetilde{\phi}_j + \rho_j$ for $j > n-K$.
The results are displayed in Figure \ref{outerrandfig}.

\item {\bf Example 5}.
We chose two different sets of angles. For the first set,
$\phi_j^{(1)}$ was chosen uniformly distributed between $0$ and $1$.
For the second set, we chose
$\phi_1^{(2)}=0,\phi_n^{(2)}=1$ and for every odd $j$ the angle
$\phi_j^{(2)}$ is linearly interpolated between the extremes. For even $j$
we linearly interpolated $\phi_j^{(2)}$ between $0$ and $2$.
The results in Figure \ref{twospiralsfig} show that the residual
makes almost no progress at every other iteration step. The residual
for the second set of angles closely follows, but makes more even progress
at all iteration steps.
\end{itemize}

\begin{figure}
\begin{center}
\includegraphics{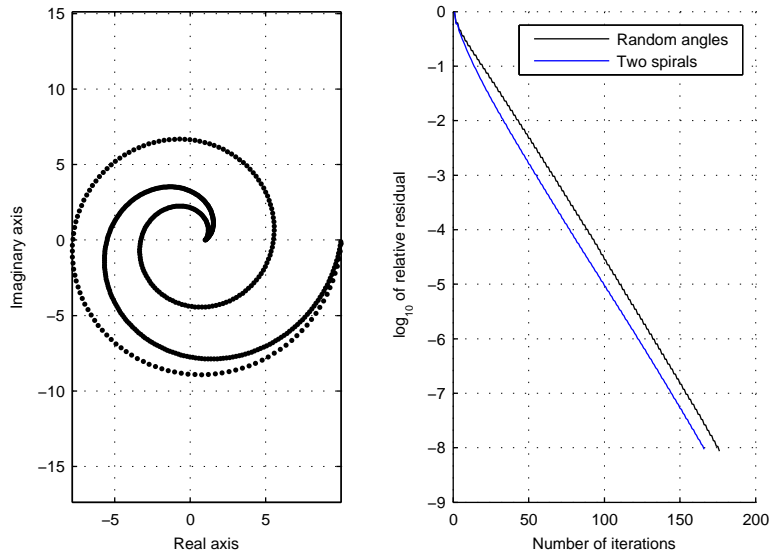}
\caption{Example 5. Only the diagonal of $\alp D$ in the two spirals case
is plotted on the left panel.}
\label{twospiralsfig}
\end{center}
\end{figure}

\end{document}